\documentclass{amsart}
\usepackage{amsthm,amscd,amssymb,verbatim,epsf,amsmath,amsfonts,mathrsfs,graphicx,dirtytalk,pdfpages,mdframed,tikz-cd,xy}
\usepackage[colorlinks=true,linkcolor=blue,citecolor=blue]{hyperref}

\theoremstyle{plain}
\newtheorem{Thm}{Theorem}[section]
\newtheorem{Cor}[Thm]{Corollary}

\newtheorem{Prop}[Thm]{Proposition}
\theoremstyle{definition}
\newtheorem{Def}[Thm]{Definition}

\newtheorem{Ex}[Thm]{Example}

\theoremstyle{remark}

\begin{document}
\title[Polynomials and Umbilics]%
   {Roots of Polynomials and Umbilics of Surfaces}
\author{Brendan Guilfoyle}
\address{Brendan Guilfoyle\\
          School of STEM \\
          Munster Technological University\\
          Tralee  \\
          Co. Kerry \\
          Ireland.}
\email{brendan.guilfoyle@mtu.ie}
\author{Wilhelm Klingenberg}
\address{Wilhelm Klingenberg\\
 Department of Mathematical Sciences\\
 University of Durham\\
 Durham DH1 3LE\\
 United Kingdom}
\email{wilhelm.klingenberg@durham.ac.uk }

\date{\today}

\begin{abstract}
For certain polynomials we relate the number of roots inside the unit circle with the index of a non-degenerate isolated umbilic point on a real analytic surface in Euclidean 3-space. 

In particular, for $N>0$ we prove that for a certain ($N+2$)-real dimensional family of complex polynomials of degree $N$, the number of roots inside the unit circle is less than or equal to $1+N/2$. This bound is established as follows. 

From the polynomial we construct a convex real analytic surface containing an isolated umbilic point, such that the index of the umbilic point is determined by the number of roots of the polynomial that lie inside the unit circle. The bound on the number of roots then follows from Hamburger's bound on the index of an isolated umbilic point on a convex real analytic  surface. 

The class of polynomials that arise are those with self-inversive second derivative. Thus the number of roots inside the unit circle is proven to be bounded for a polynomial with self-inversive second derivative.

\end{abstract}

\maketitle

\section{Introduction}

Almost two centuries ago Cauchy established that all zeros of a complex polynomial must lie inside a circle of a radius that can be determined by the moduli of the polynomial's coefficients \cite{cauchy}. Over the intervening years a vast number of relations have been found between polynomial roots lying inside a circle or annulus, and the coefficients of the polynomial \cite{dattandgov} \cite{kumar} \cite{RaS}. 

The determination of the location of the zeros of a given polynomial relative to the unit circle in the complex plane arises in questions of stability in many areas of applied mathematics, including numerical integration, difference equations, control theory and economic growth models \cite{duffin} \cite{miller} \cite{richandmort} \cite{samuelson}. 

Polynomials whose zero set is invariant under inversion in the unit circle $\zeta \mapsto 1/\bar{\zeta}$ are called {\em self-inversive}. The roots of self-inversive polynomials have been considered for numerous reasons \cite{BaM52} \cite{OhaR74} \cite{TSS}, including their relationship to Ramanujan polynomials \cite{V17} and codes \cite{JaS18}. 

In this paper we present a relationship between the number of zeros contained inside the unit circle of polynomials with self-inversive second derivative and the index of isolated umbilic points on real analytic convex surfaces in ${\mathbb R}^3$. A deep result of Hamburger \cite{ham} bounds the index of an isolated umbilic point on any real analytic convex surface, which yields a bound on the number of zeros inside the unit circle. 

Umbilic points have many subtle properties - they can control the evolution of surfaces under curvature flows \cite{gak4}, while Hamburger's index bound does not hold for even small perturbations of the ambient Euclidean metric \cite{g21}. 

It is a remarkable fact that the polynomial constraints that arise from the Lagrangian condition in the space of oriented affine lines in ${\mathbb R}^3$ at a non-degenerate umbilic point are equivalent to the second derivative of the lowest order polynomial being self-inversive.

The main result of this paper is:

\vspace{0.1in}

\noindent{\bf Main Theorem}:

{\it
Let $P_N$ be a polynomial with self-inversive second derivative and suppose that none of the roots of $P_N$ lies on the unit circle. Then the number of roots (counted with multiplicity) of $P_N$ inside the unit circle is less than or equal to $\lfloor N/2\rfloor+1$.
}

\vspace{0.1in}

One way to view this result is in the spirit of a converse to the fundamental Gauss-Lucas theorem \cite{Marden} in which the zeros of the first derivative of a polynomial are restricted by the zeros of the polynomial. Here, however, by methods of differential geometry, the locations of the zeros of the second derivative restrict the zeros of the polynomial. It is also worth noting that the result is sharp in that for any integer $N\geq 3$ there exists a polynomial $P_N$ satisfying the hypothesis of the Main Theorem for which the bound $\lfloor N/2\rfloor+1$ for the number of zeros inside the unit circle is attained.

We prove this result by showing in Theorem \ref{t:main} that, given a polynomial $P_N$ satisfying the self-inversive conditions (\ref{e:rels}), there exists a convex real analytic surface in ${\mathbb R}^3$ with an isolated umbilic point which has index $I=K-N/2$, where $K$ is the number of zeros of $P_N$ lying inside the unit circle. Thus, the local umbilic index bound of \cite{ham}, which states that $I\leq 1$, yields the stated result. A different approach to self-inversive polynomials and the support function can be found in the thesis \cite{Mayost}.

In the next section we give the background on the construction of convex surfaces, using Lagrangian surfaces in the space of oriented affine lines in ${\mathbb R}^3$ - further details can be found in \cite{gak2}.  Section 3 contains the proof of the Main Theorem and in the final section we illustrate our techniques by constructing explicit examples. 

\vspace{0.2in}

\section{Real Analytic Convex Surfaces in ${\mathbb R}^3$}

\begin{Def}
Let ${\mathbb L}$ be the set of oriented (affine) lines in Euclidean 3-space ${\mathbb R}^3$. 
\end{Def}
\begin{Def}
Let $\Phi:TS^2\rightarrow{\mathbb L}$ be the map that identifies ${\mathbb L}$ with the tangent bundle to the unit 2-sphere in Euclidean ${\mathbb R}^3$, by parallel translation. This bijection gives ${\mathbb L}$ the structure of a differentiable 4-manifold \cite{gak5}.
\end{Def}

Let ($\xi,\eta$) be holomorphic coordinates on $TS^2$, where $\xi$ is obtained by stereographic projection from the south pole onto the plane through the equator, and we identify ($\xi,\eta$) with the vector 
\[
\eta\frac{\partial}{\partial \xi}+\overline{\eta}\frac{\partial}{\partial \overline{\xi}}\in T_\xi S^2.
\]

\begin{Prop}\cite{gak2}
The map $\Phi$ takes ($\xi,\eta$)$\in TS^2$ to the oriented line given by
\begin{equation}\label{e:coord1}
z=\frac{2(\eta-\overline{\eta}\xi^2)+2\xi(1+\xi\overline{\xi})r}{(1+\xi\overline{\xi})^2}
\end{equation}
\begin{equation}\label{e:coord2}
t=\frac{-2(\eta\overline{\xi}+\overline{\eta}\xi)+(1-\xi^2\overline{\xi}^2)r}{(1+\xi\overline{\xi})^2},
\end{equation}
where $z=x^1+ix^2$, $t=x^3$, ($x^1,x^2,x^3$) are Euclidean coordinates on ${\mathbb R}^3={\mathbb{C}}\oplus{\mathbb{R}}$ and $r$ is an affine parameter along the line.
\end{Prop} 

An {\it oriented line congruence} (or line congruence for short) is a surface $\Sigma\subset{\mathbb L}$. An oriented surface $S$ in ${\mathbb R}^3$ gives rise to a line congruence by its oriented normal lines. Conversely a line congruence is called {\it Lagrangian} if it is orthogonal to a (possibly singular) surface in ${\mathbb R}^3$. The term refers to the canonical symplectic structure on ${\mathbb L}$ which vanishes on such a surface - see \cite{gak5} for further details of the geometric background.

\vspace{0.1in}

\begin{Prop}
The oriented normals of a convex surface in ${\mathbb R}^3$ form a surface $\Sigma\subset{\mathbb L}$ which is the graph of a local section of the bundle $\pi:{\mathbb L}\rightarrow S^2$.
\end{Prop}
\vspace{0.1in}

Such congruences are given by $\eta=F(\xi,\bar{\xi})$, where $F$ is a complex-valued function on an open neighbourhood of $S^2$. For such local sections, a necessary and sufficient condition for the line congruence to be Lagrangian is: 
\vspace{0.1in}

\begin{Prop}
The graph of a local section of $\pi:{\mathbb L}\rightarrow S^2$ is Lagrangian iff the defining function $F$ satisfies
\begin{equation}\label{e:int}
\partial\left(\frac{F}{(1+\xi\bar{\xi})^2}\right)=
\bar{\partial}\left(\frac{\bar{F}}{(1+\xi\bar{\xi})^2}\right),
\end{equation}
where, here and throughout, we denote ${\textstyle{\frac{\partial}{\partial\xi}}}$ by $\partial$. In particular, there exists a real function $r:{\mathbb C}\rightarrow{\mathbb R}$, called the support function, satisfying
\begin{equation}\label{e:suppfunc}
\bar{\partial}r=\frac{2F}{(1+\xi\bar{\xi})^2}.
\end{equation}
\end{Prop}
\vspace{0.1in}

Given a support function $r$ and the associated $F$ defined by (\ref{e:suppfunc}), equations (\ref{e:coord1}) and (\ref{e:coord2}) yield a parameterised surface in ${\mathbb R}^3$. Thus we obtain a surface in 3-space parameterised by the inverse of its Gauss map. Note that $r$ is defined up to an additive constant - replacing $r$ by $r+C$ moves the surface along its normal lines to the parallel surface.

For a convex real analytic surface the support function is real analytic (as it can be defined by projection of the position vector onto the real analytic normal vector) and so the graph function $F$ defined by equation (\ref{e:suppfunc}) is also real analytic. Thus we expand it in a power series about $\xi=0$:
\begin{equation}\label{e:powerseries}
F=\sum_{n,m=0}^\infty A_{n,m}\xi^n\bar{\xi}^m,
\end{equation}
for $A_{n,m}\in{\mathbb{C}}$. 

\vspace{0.1in}

\begin{Prop}
The Lagrangian conditions (\ref{e:int}) restrict the coefficients of the power series by
\begin{equation}\label{e:lag1}
A_{1,0}=\overline{A}_{1,0}
\end{equation}
\begin{equation}\label{e:lag2}
(n+1)A_{n+1,0}=\overline{A}_{1,n}-2\overline{A}_{0,n-1}
\end{equation}
\begin{equation}\label{e:lag3}
(n+1)A_{n+1,m}+(n-2)A_{n,m-1}
     =(m+1)\overline{A}_{m+1,n}+(m-2)\overline{A}_{m,n-1},
\end{equation}
where $n,m\geq 1$.
\end{Prop}
\begin{proof}
Suppose that
\[
F=\sum_{n,m=0}^\infty A_{n,m}\xi^n\bar{\xi}^m,
\]
and differentiating
\[
\partial F=\sum_{n,m=0}^\infty nA_{n,m}\xi^{n-1}\bar{\xi}^m.
\]
Then
\begin{align}
(1+\xi\bar{\xi})\partial F-2\bar{\xi}F&= \sum_{n,m=0}^\infty nA_{n,m}\xi^{n-1}\bar{\xi}^m+(n-2)A_{n,m}\xi^{n}\bar{\xi}^{m+1}\nonumber\\
&=\sum_{n=1}^\infty nA_{n,0}\xi^{n-1}+\sum_{m=1}^\infty  A_{1,m}\bar{\xi}^m+\sum_{n=2}^\infty\sum_{m=1}^\infty  nA_{n,m}\xi^{n-1}\bar{\xi}^m\nonumber\\
&\qquad  -\sum_{m=1}^\infty 2A_{0,m-1}\bar{\xi}^m+\sum_{n,m=1}^\infty (n-2)A_{n,m-1}\xi^n\bar{\xi}^m\nonumber\\
&=A_{10}+\sum_{m=1}^\infty (m+1)A_{m+1,0}\xi^{m}+\sum_{m=1}^\infty  (A_{1,m}-2A_{0,m-1})\bar{\xi}^m\nonumber\\
&\qquad +\sum_{n,m=1}^\infty [(n+1)A_{n+1,m}+(n-2)A_{n,m-1}]\xi^n\bar{\xi}^m.\nonumber
\end{align}
The result then follows from noting that the condition (\ref{e:int}) is equivalent to
\[
(1+\xi\bar{\xi})\partial F-2\bar{\xi}F=(1+\xi\bar{\xi})\bar{\partial} \bar{F}-2\xi\bar{F},
\]
and comparing terms in the above expression.
\end{proof}

\vspace{0.1in}

We now rearrange these conditions into a more convenient form.

\begin{Prop}\label{p:lagcon}
The Lagrangian conditions can be rewritten
\[
\begin{matrix}
A_{n,n-1}=\bar{A}_{n,n-1} & {\mbox{I}}\\
2(n-2)A_{0,n-2}-2(n-1)A_{1,n-1}+2nA_{2,n}=n(n+1)\bar{A}_{n+1,1} &II \\
(n+2)A_{n+2,2}+(n-1)A_{n+1,1}=3\bar{A}_{3,n+1}  &III \\
(n+2)A_{n+2,m+1}+(n-1)A_{n+1,m}=(m+2)\bar{A}_{m+2,n+1}+(m-1)\bar{A}_{m+1,n}  &IV\\
nA_{n,0}=\bar{A}_{1,n-1}-2\bar{A}_{0,n-2}.  & V\\
\end{matrix},
\]
for $n>1, m<n$.
\end{Prop}

\vspace{0.1in}

From here on, we assume that  $S$  is a convex real analytic surface in ${\mathbb R}^3$ and the oriented normals are given by $\xi\rightarrow(\xi,\eta=F(\xi,\bar{\xi}))$, with power series expansion (\ref{e:powerseries}) satisfying I to V of Proposition \ref{p:lagcon}.

We are interested in the umbilic points on $S$, that is, the points where the second fundamental form of $S$ has a double eigenvalue. In our formalism, these points can be characterised as follows:

\vspace{0.1in}

\begin{Prop}\cite{gak2}
A point $p$ on $S$ is umbilic iff $\bar{\partial}F(\gamma)=0$, where $\gamma$ is the oriented normal to $S$ at $p$.
\end{Prop}

\vspace{0.1in}

Suppose that $S$ has an isolated umbilic point situated at the point with $\xi=0$, and define the polar coordinates $(R,\theta)$
by $\xi=Re^{i\theta}$.

\begin{Def}
The {\it index} $I$ of the isolated umbilic point at 0 is
\[
I=\lim_{R\rightarrow 0}{\textstyle{\frac{1}{8\pi i}}}
      \int_0^{2\pi}\frac{\partial}{\partial \theta}\ln\left(\frac{\bar{\partial}F}{\partial\bar{F}}\right)d\theta.
\]
\end{Def}
\vspace{0.1in}

\begin{Ex}
For the triaxial ellipsoid with semi-axes $a_1$, $a_2$ and $a_3$ the support function is
\[
r=\left[a_1\left(\frac{\xi+\overline{\xi}}{1+\xi\overline{\xi}}\right)^2
        -a_2\left(\frac{\xi-\overline{\xi}}{1+\xi\overline{\xi}}\right)^2
        +a_3\left(\frac{1-\xi\overline{\xi}}{1+\xi\overline{\xi}}\right)^2\right]^{\scriptstyle{\frac{1}{2}}},
\]
and the Lagrangian section is
\[
F=\frac{a_1(\xi+\overline{\xi})(1-\xi^2)+a_2(\xi-\overline{\xi})(1+\xi^2)
    -2a_3\xi(1-\xi\overline{\xi})}
     {2\left[a_1(\xi+\overline{\xi})^2-a_2(\xi-\overline{\xi})^2
     +a_3(1-\xi\overline{\xi})^2\right]^{\scriptstyle{\frac{1}{2}}}}.
\]
For $a_1=a_2\neq a_3$ the surface is the rotationally symmetric ellipsoid (the $x^3-$axis being the axis of symmetry). For $a_1=a_2=a_3$ the surface is the round 2-sphere about the origin of radius $\sqrt{a_1}$ which is totally umbilic.

For the rotationally symmetric ellipsoid $a_1=a_2=a\neq a_3$
\[
r=\frac{[4aR^2+a_3(1-R^2)^2]^{\scriptstyle{\frac{1}{2}}}}{1+R^2}
\qquad\qquad
F=\frac{(a-a_3)R(1-R^2)e^{i\theta}}{[4aR^2+a_3(1-R^2)^2]^{\scriptstyle{\frac{1}{2}}}},
\]
in polar coordinates $\xi=Re^{i\theta}$.

Computing the derivative of $F$ 
\[
\bar{\partial}F=\frac{2a(a-a_3)R^2(1+R^2)e^{2i\theta}}{[4aR^2+a_3(1-R^2)^2]^{\scriptstyle{\frac{3}{2}}}},
\]
and so the umbilics lie at $R=0$ and $R\rightarrow\infty$.

To compute the index of the umbilic at $R=0$ we find
\[
I=\lim_{R\rightarrow 0}{\textstyle{\frac{1}{8\pi i}}}
      \int_0^{2\pi}\frac{\partial}{\partial \theta}\ln\left(\frac{\bar{\partial}F}{\partial\bar{F}}\right)d\theta={\textstyle{\frac{1}{2\pi }}}\int_0^{2\pi}d\theta=1.
\]
On the other hand, for the triaxial ellipsoid $a_1\neq a_2\neq a_3$, a similar calculation shows that there are 4 umbilic points which lie at $\theta=0$ and
\[
R=\pm\left[-{{\frac{(a_1+a_2)a_3-2a_1a_2}{(a_1-a_2)a_3}}}\pm\left[\left({{\frac{(a_1+a_2)a_3-2a_1a_2}{(a_1-a_2)a_3}}}\right)^2-1\right]^{\scriptstyle{\frac{1}{2}}}\right]^{\scriptstyle{\frac{1}{2}}}.
\]
Denote either of these values by $R=R_0$. A rotation
\[
\xi\rightarrow\frac{\xi-R_0}{1+R_0\xi} \qquad\qquad F\rightarrow \frac{(1+R_0^2)F}{(1+R_0\xi)^2},
\]
moves this umbilic to the origin $R=0$ and a lengthy computation yields
\[
I=\lim_{R\rightarrow 0}{\textstyle{\frac{1}{8\pi i}}}
      \int_0^{2\pi}\frac{\partial}{\partial \theta}\ln\left(\frac{\bar{\partial}F}{\partial\bar{F}}\right)d\theta
={\textstyle{\frac{1}{4\pi }}}\int_0^{2\pi}d\theta
={\textstyle{\frac{1}{2 }}}.
\]

\end{Ex}
\vspace{0.2in}

For any section $F$, define the complex-valued functions $G_k(\theta)$ by the polar decomposition of $\bar{\partial}F$:
\[
\bar{\partial}F=\sum_{k=0}^\infty G_k(\theta)e^{-ik\theta}R^k.
\]
Then
\[
G_k(\theta)=\sum_{n=0}^k(k-n+1)A_{n,k-n+1}e^{2in\theta},
\]
and, in particular, $G_0=A_{0,1}=\bar{\partial}F(0)$, which vanishes iff $\xi=0$ is an umbilic point of $S$.

\begin{Def}
An isolated umbilic point at $\xi=0$ is said to be {\it non-degenerate of order} $N$ if
\[
G_k(\theta)=0 \qquad {\mbox{for}}\;k<N
\qquad\qquad {\mbox{and}} \qquad\qquad
G_N(\theta)\neq0 \qquad {\mbox{for all}}\;\theta.
\]
For such an umbilic point, define the degree $N$ polynomial
\[
P_N=\sum_{n=0}^N(N-n+1)\frac{A_{n,N-n+1}}{A_{N,1}}\zeta^{n}.
\]
Here and throughout we assume that $A_{N,1}\neq0$. 
\end{Def}

\vspace{0.1in}
\begin{Prop}
For a non-degenerate isolated umbilic point of order N, the polynomial $P_N$ does not have any zeros on the unit circle $|\zeta|=1$ and the index
of the umbilic point is
\[
I=K-N/2,
\]
where $K$ is the number of zeros of $P_N$ that lie inside the unit circle.
\end{Prop}
\begin{proof}
On the unit circle $\zeta=e^{2i\theta}$ we have
\[
P_N=\frac{G_N(\theta)}{A_{N,1}}\neq0,
\]
which proves the first statement. For the second statement
\begin{align}
I&=\lim_{R\rightarrow 0}{\textstyle{\frac{1}{8\pi i}}}
      \int_0^{2\pi}\frac{\partial}{\partial \theta}\ln\left(\frac{\bar{\partial}F}{\partial\bar{F}}\right)d\theta\nonumber\\
&={\textstyle{\frac{1}{8\pi i}}}\int_0^{2\pi}\frac{\partial}{\partial \theta}\ln\left(\frac{G_N(\theta)e^{-iN\theta}}{\bar{G}_N(\theta)e^{iN\theta}}\right)d\theta\nonumber\\
&={\textstyle{\frac{1}{8\pi i}}}\int_0^{2\pi}\frac{\partial}{\partial \theta}\ln\left(\frac{P_N}{\bar{P}_N}\right)d\theta
    -N/2\nonumber\\
&=K-N/2,\nonumber
\end{align}
as claimed.
\end{proof}

\vspace{0.1in}

However, we also have
\vspace{0.1in}

\begin{Thm}\cite{ham}
The index of an isolated umbilic point on a real analytic convex surface in ${\mathbb R}^3$ is less than or equal to 1.
\end{Thm}

\vspace{0.1in}

As a corollary, conclude 

\vspace{0.1in}

\begin{Cor}
For a non-degenerate umbilic point of order $N$ on a real analytic surface, the number of zeros of the polynomial $P_N$ that lie
inside the unit circle is less than or equal to $1+N/2$. 
\end{Cor}
\vspace{0.1in}

Of course the polynomial $P_N$ is not arbitrary: its coefficients satisfy the relations which are contained in the Lagrangian 
conditions I to V of Proposition \ref{p:lagcon}. To prove our Main Theorem we reverse the preceding argument: given a polynomial 
$P_N$ satisfying the order $N$ Lagrangian conditions, we construct a convex surface in ${\mathbb R}^3$ with isolated umbilic point
for which $P_N$ is the lowest order term in $R$ of $\bar{\partial} F$. This we do in the next section.

\vspace{0.2in}

\section{Proof of the Main Theorem}

Our first task is to determine the restrictions placed on the coefficients of the polynomial $P_N$ by the Lagrangian conditions I to V.
To understand these conditions, define the Fourier decomposition of $\bar{\partial}F$ about the origin by
\[
\bar{\partial}F=\sum_{k=-\infty}^\infty\gamma_k(R)e^{ik\theta},
\]
where
\[
\gamma_k(R)=\frac{1}{2\pi}\int_{0}^{2\pi}\bar{\partial}F\;e^{-ik\theta}d\theta.
\]
In Figure 1 we decompose each of these Fourier coefficients in powers of $R$ and indicate the Lagrangian conditions placed on the
coefficients $A_{n,m}$ by enclosing related terms by a box. The four Lagrangian conditions I to IV each has a distinctive shape which 
enables us to trace this net of relations and assist in calculations. 

Condition V does not introduce any relationship on the coefficients of $\bar{\partial}F$ as it
involves $A_{n,0}$, which is annihilated by $\bar{\partial}$. As such, we can view condition V as defining the coefficients $A_{n,0}$ in the full
power series after we have determined the other coefficients.  

\vspace{0.1in}

\begin{Prop}
For the polynomial $P_N$ associated with a non-degenerate umbilic point of order $N$ the Lagrangian conditions I to IV mean that the coefficients
must satisfy
\begin{equation}\label{e:lag1a}
(n+2)A_{n+2,N-n-1}=(N-n)\bar{A}_{N-n,n+1} \qquad {\mbox{for}}\; 1\leq n\leq \lfloor N/2\rfloor-1
\end{equation}
\begin{equation}\label{e:lag2a}
2A_{2,N-1}=N\bar{A}_{N,1}.
\end{equation}
\end{Prop}
\begin{proof}
Given $N$, the $R^N$ column of Figure 1 lists the lowerest order non-zero coefficients of $\bar{\partial}F$, namely those that enter into the
definition of $P_N$. The Lagrangian condition that these coefficients must satisfy are indicated by the boxes in which they lie. Equations
(\ref{e:lag1a}) and (\ref{e:lag2a}) are obtained by inspection of the boxes that contain only these non-zero terms and lower order
terms, which by assumption vanish.
\end{proof}
\vspace{0.1in}

We now show that these relations are precisely the equations for the second derivative of $P_N$ to be self-inversive. In more detail, let $P_N$ be a polynomial over ${\mathbb C}$ of degree $N$: 
\begin{align*}
P_N(\zeta)&=\prod_{n=1}^N(\zeta-\zeta_n)\\
 &=\zeta^N+N\Delta_1^N\zeta^{N-1}+{\textstyle{\frac{N(N-1)}{2}}}\Delta_2^N\zeta^{N-2}+...+\;\binom{N}{n}\Delta_n^N\zeta^{N-n}+...+\Delta_N^N,
\end{align*}
where the weighted symmetric functions of the roots $ \{\zeta_n\}_{n=1}^N$ are defined by
\[
{\textstyle{\binom{N}{n}}}\Delta_n^N=\sum_{1\leq i_1<...< i_n\leq N}\zeta_{i_1}\zeta_{i_2}...\zeta_{i_n}.
\]
Consider the following conditions on the coefficients of $P_N$:
\begin{equation}\label{e:rels}
  |\Delta_{N-2}^N|=1
\qquad\qquad
\Delta_{n}^N=\overline{\Delta_{N-2-n}^N}\Delta_{N-2}^N \qquad{\mbox{ for}}\;1\leq n\leq \lfloor N/2\rfloor-1,
\end{equation} 
where $\lfloor x\rfloor$ denotes the largest integer smaller than $x$. Note that these are $N-2$ real conditions on the coefficients of $P_N$,  which give rise to a real $N+2$ dimensional family of degree $N$ polynomials. 

\vspace{0.1in}
\begin{Cor}
For the polynomial $P_N$ associated with a non-degenerate umbilic point of order $N$ the Lagrangian conditions I to IV mean that 
the symmetric functions $\Delta^N_n$ must satisfy equations (\ref{e:rels}).
\end{Cor}
\begin{proof}
We have the definitions
\[
P_N=\sum_{k=0}^N(k+1)\frac{A_{N-k,k+1}}{A_{N,1}}\zeta^{N-k}=\sum_{k=0}^N\binom{N}{k}\;\Delta_k^N\zeta^{N-k},
\]
and so we have the relations
\begin{equation}\label{e:1}
\binom{N}{k}\Delta_k^N=(k+1)\frac{A_{N-k,k+1}}{A_{N,1}}.
\end{equation}
For $k=N-2$ this reads
\[
{\textstyle{\frac{N(N-1)}{2}}}\Delta^N_{N-2}=(N-1)\frac{A_{2,N-1}}{A_{N,1}}.
\]
Thus, with the aid of equation (\ref{e:lag2a}) we have
\[
\Delta^N_{N-2}=\frac{2A_{2,N-1}}{NA_{N,1}}=\frac{\bar{A}_{N,1}}{A_{N,1}}.
\]
Therefore $|\Delta^N_{N-2}|=1$, as claimed.

On the other hand, utilising equation (\ref{e:lag1a}) we deduce that
\[
\binom{N}{k}\Delta_k^N=(k+1)\frac{A_{N-k,k+1}}{A_{N,1}}={\textstyle{\frac{(k+1)(k+2)}{N-k}}}\frac{\bar{A}_{k+2,N-k-1}}{A_{N,1}}.
\]
Now applying equation (\ref{e:1}) again with $k$ replaced by $N-k-2$, we get
\[
\binom{N}{k}\Delta_k^N={\textstyle{\frac{(k+1)(k+2)}{(N-k)(N-k-1)}}}\binom{N}{N-k-2}\frac{\bar{A}_{N,1}}{A_{N,1}}\bar{\Delta}^N_{N-k-2},
\]
which, after some cancellation, reads
\[
\Delta_k^N=\Delta^N_{N-2}\overline{\Delta_{N-k-2}^N},
\]
as claimed.
\end{proof}
\vspace{0.1in}
\begin{Prop}
    The conditions (\ref{e:rels}) on a polynomial $P_N$ are equivalent to the second derivative $P_N''$ being self-recursive.
\end{Prop}
\begin{proof}
Starting with the definition
\[
P_N(\zeta)=\sum_{n=0}^{N}\binom{N}{n}\Delta_n^N\zeta^{N-n},
\]
differentiate twice to get that
\[
P_N''(\zeta)=\sum_{n=0}^{N-2}\binom{N}{n}(N-n)(N-n-1)\Delta_n^N\zeta^{N-n-2}=\sum_{n=0}^{N-2}\frac{N!}{n!(N-n-2)!}\Delta_n^N\zeta^{N-n-2}.
\]
The self-inversive property for the polynomial $P_n''$ of degree $N-2$ is
\[
P_n''(\zeta)=e^{i\beta}\zeta^{N-2}\overline{P_n''\left({\textstyle{\frac{1}{\overline{\zeta}}}}\right)},
\]
for some real constant $\beta$. In terms of the expansion above we have
\[
\sum_{n=0}^{N-2}\frac{N!}{n!(N-n-2)!}\Delta_n^N\zeta^{N-n-2}=e^{i\beta}\sum_{m=0}^{N-2}\frac{N!}{m!(N-m-2)!}\Delta_m^N\zeta^{m}
\]
and so comparing terms we have $m=N-n-2$ and
\[
\Delta_n^N=e^{i\beta}\overline{\Delta_{N-n-2}^N}.
\]
For $n=0$ we have $1=\Delta_0^N=e^{i\beta}\overline{\Delta_{N-2}^N}$ and so
\[
e^{i\beta}={\Delta_{N-2}^N} \qquad \Delta_n^N={\Delta_{N-2}^N}\overline{\Delta_{N-n-2}^N}
\]
which are equations (\ref{e:rels}) as claimed.

\end{proof}
\vspace{0.1in}

Now, given a polynomial $P_N$ satisfying relations (\ref{e:rels}) we seek to construct a convex real analytic surface in ${\mathbb R}^3$ with umbilic point at $\xi=0$. We do this by first using the coefficients defined by $P_N$ as terms of order $N$ in the power series of $\bar{\partial}F$, and then adding further terms of higher order in such a way that the full Lagrangian condition holds. By moving to higher order terms we ensure that the index of the umbilic point resides in the lowest order term, namely $P_N$. 

Such an extension will be far from unique, but we seek to extend in such a way that minimises the number of non-zero terms introduced. Finally, we must ensure that the resulting umbilic point at $\xi=0$ is indeed isolated, and hence we can apply Hamburger's result to prove the stated Theorem.

Once again Figure 1 is useful as a guide to the equations we are dealing with. Suppose then that we have non-zero terms in the $R^N$ column of $\bar{\partial}F$ (namely, those given by $P_N$) and all terms of lower degree are zero. We will see that it is only necessary to go to terms of order $R^{N+2}$ to complete the Lagrangian conditions. In this we identify 6 situations of overlapping equations that must dealt with separately.

\setlength{\epsfxsize}{4.5in}
\begin{center}
\includegraphics[width=80mm]{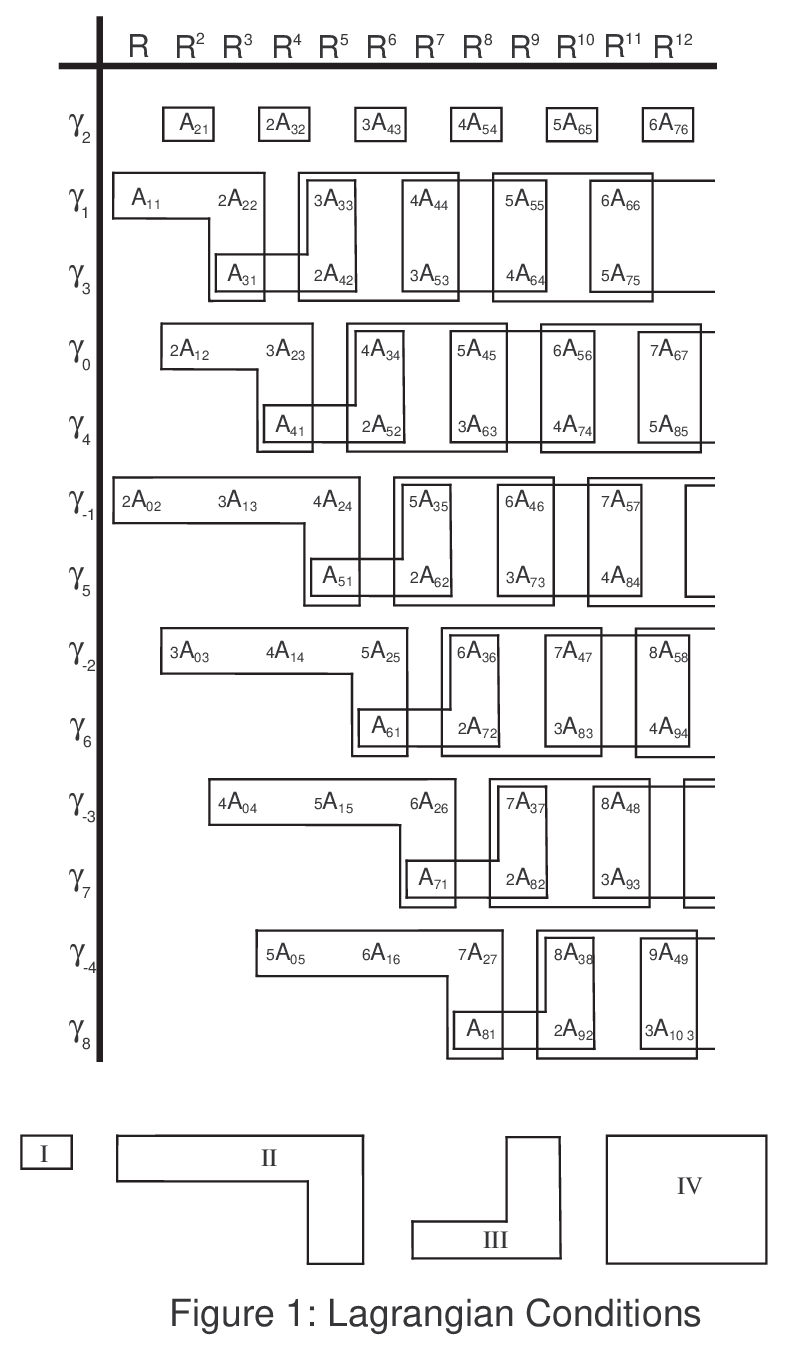}
\end{center}

\noindent{\bf Case 1: IV IV IV}

Thus we have the Lagrangian condition
\[
(n+2)A_{n+2,m+1}=(m+2)\bar{A}_{m+2,n+1},
\]
and the two higher order equations
\[
(n+3)A_{n+3,m+2}+nA_{n+2,m+1}=(m+3)\bar{A}_{m+3,n+2}+m\bar{A}_{m+2,n+1}
\]
\[
(n+1)A_{n+3,m+2}=(m+1)\bar{A}_{m+3,n+2}.
\]
The first and third of these, rearranged to
\[
A_{m+2,n+1}=\frac{n+2}{m+2}\bar{A}_{n+2,m+1} \qquad\qquad A_{m+3,n+2}=\frac{n+1}{m+1}\bar{A}_{n+3,m+2},
\]
and substituted in the second, yield 
\[
A_{n+3,m+2}=\frac{m+1}{m+2}A_{n+2,m+1},
\]
and so 
\[
A_{m+3,n+2}=\frac{n+1}{m+2}\bar{A}_{n+2,m+1}.
\]
With these choices the power series is truncated.

\vspace{0.1in}

\noindent{\bf Case 2: III IV IV}

First we have the Lagrangian condition at order $N$:
\[
(N-1)A_{N-1,2}=3\bar{A}_{3,N-2},
\]
and two order $N+2$ conditions:
\[
NA_{N,3}+(N-3)A_{N-1,2}=4\bar{A}_{4,N-1}+\bar{A}_{3,N-2}
\]
\[
(N-2)A_{N,3}=2\bar{A}_{4,N-1}.
\]
Combining these as before we get
\[
A_{N-1,2}=\frac{3}{N-1}\bar{A}_{3,N-2}
\qquad
A_{4,N-1}=\frac{N-2}{N-1}A_{3,N-2}
\qquad
A_{N,3}=\frac{2}{N-1}\bar{A}_{3,N-2}.
\]

\vspace{0.1in}

\noindent{\bf Case 3: II III IV}

The Lagrangian condition at order $N$ is
\[
2A_{2,N-1}=N\bar{A}_{N,1},
\]
and the other conditions are
\[
(N+1)A_{N+1,2}+(N-2)A_{N,1}=3\bar{A}_{3,N}
\]
\[
(N-1)A_{N+1,2}=\bar{A}_{3,N}.
\]
Solving these we have
\[
A_{N,1}=\frac{2}{N}\bar{A}_{2,N-1}
\qquad
A_{3,N}=\frac{N-1}{N}A_{2,N-1}
\qquad
A_{N+1,2}=\frac{1}{N}\bar{A}_{2,N-1}.
\]

\vspace{0.1in}

\noindent{\bf Case 4: II}

Here we have exhausted the Lagrangian condition, but must accommodate the tail of the polynomial $P_N$. That is
\[
-NA_{1,N}+(N+1)A_{2,N+1}=0,
\]
or
\[
A_{2,N+1}=\frac{N}{N+1}A_{1,N}.
\]
The other term is
\[
(N+1)A_{0,N+1}-(N+2)A_{1,N+2}=0,
\]
or
\[
A_{1,N+2}=\frac{N+1}{N+2}A_{0,N+1}.
\]

\vspace{0.1in}

\noindent{\bf Case 5: V}

As mentioned earlier, this Lagrangian condition does not enter into the relations enjoyed by the coefficients of $\bar{\partial}F$, but is taken as
a definition of the non-vanishing holomorphic part of $F$. In particular, we have
\[
A_{N+1,0}=\frac{1}{N+1}\bar{A}_{1,N},
\]
and
\[
(N+3)A_{N+3,0}=\bar{A}_{1,N+2}-2\bar{A}_{0,N+1}=\frac{N+1}{N+2}\bar{A}_{0,N+1}-2\bar{A}_{0,N+1},
\]
or
\[
A_{N+3,0}=-\frac{1}{N+2}\bar{A}_{0,N+1}.
\]

\vspace{0.1in}

\noindent{\bf Case 6: I}

When $N=2l+2$ is even, we have the Lagrangian condition
\[
A_{l+2,l+1}=\bar{A}_{l+2,l+1}.
\]
As this involves only one coefficient of the power series it is not strictly necessary to relate it to any other term to maintain the
Lagrangian condition. However, in order to ensure that the umbilic on the surface is isolated, we impose the final following condition
\[
A_{l+3,l+2}=\frac{l}{l+2}A_{l+2,l+1}.
\]   

Combining all of the preceding cases, we now establish:

\vspace{0.1in}

\begin{Thm}\label{t:main}
Given a polynomial $P_N$ satisfying the relations (\ref{e:rels}), choose $A_{N,1}\in{\mathbb C}^*$ and define the
complex numbers $\{A_{n,N-n+1}\}_{n=0}^{N-1}$ by
\[
A_{N-n,n+1}={\textstyle{\frac{1}{n+1}}}\binom{N}{n}A_{N,1}\Delta_n^N,
\]
for $n=1,2,...,N$.

Then, for $N=2l+1$ the following function defines a convex real analytic surface $S$ in ${\mathbb R}^3$:
\[
F=\sum_{n=0}^{l+1}A_{n,N-n+1}(1+{\textstyle{\frac{N-n+1}{N-n+2}}}\xi\bar{\xi})\xi^{n}\bar{\xi}^{N-n+1}
     +\bar{A}_{n,N-n+1}({\textstyle{\frac{n}{N-n+2}}}+{\textstyle{\frac{n-1}{N-n+2}}}\xi\bar{\xi})\xi^{N-n+2}\bar{\xi}^{n-1},
\]
with support function
\[
r=\frac{2}{1+\xi\bar{\xi}}\sum_{n=0}^{l+1} \left({\textstyle{\frac{A_{n,N-n+1}}{N-n+2}}}\xi^{n}\bar{\xi}^{N-n+2}+{\textstyle{\frac{\bar{A}_{n,N-n+1}}{N-n+2}}}\xi^{N-n+2}\bar{\xi}^{n}\right)+C   .
\]
For $N=2l+2$ the function is
\begin{align}
F&=\sum_{n=0}^{l+1}A_{n,N-n+1}(1+{\textstyle{\frac{N-n+1}{N-n+2}}}\xi\bar{\xi})\xi^{n}\bar{\xi}^{N-n+1}+\bar{A}_{n,N-n+1}({\textstyle{\frac{n}{N-n+2}}}+{\textstyle{\frac{n-1}{N-n+2}}}\xi\bar{\xi})\xi^{N-n+2}\bar{\xi}^{n-1}\nonumber \\
&\qquad\qquad\qquad
     +A_{l+2,l+1}(1+{\textstyle{\frac{l+1}{l+2}}}\xi\bar{\xi})\xi^{l+2}\bar{\xi}^{l+1},\nonumber
\end{align}
and the support function is
\begin{align}
r&=\frac{2}{1+\xi\bar{\xi}}\sum_{n=0}^{l+1}\left( {\textstyle{\frac{A_{n,N-n+1}}{N-n+2}}}\xi^{n}\bar{\xi}^{N-n+2}+{\textstyle{\frac{\bar{A}_{n,N-n+1}}{N-n+2}}}\xi^{N-n+2}\bar{\xi}^{n}\right)
\nonumber\\
&\qquad\qquad\qquad  +\frac{2}{(1+\xi\bar{\xi})(l+2)} A_{l+2,l+1}\xi^{l+2}\bar{\xi}^{l+2}   +C .\nonumber   
\end{align}

In both cases, $S$ has an isolated umbilic point at $\xi=0$ which is non-degenerate of order $N$. 
\end{Thm}
\begin{proof}
We will show that $F$ and $r$ satisfy (\ref{e:suppfunc}), which implies the Lagrangian condition (\ref{e:int}).
For $N=2l+1$ we have:
\begin{align}
\bar{\partial}r&=\frac{2}{1+\xi\bar{\xi}}\sum_{n=0}^{l+1}\left(A_{n,N-n+1}\xi^n\bar{\xi}^{N-n+1}+{\textstyle{\frac{n}{N-n+2}}}\bar{A}_{n,N-n+1}\xi^{N-n+2}\bar{\xi}^{n-1}
      \right)\nonumber\\
&\qquad\qquad\qquad\qquad -\frac{2\xi}{(1+\xi\bar{\xi})^2}\sum_{n=0}^{l+1}\left(\frac{A_{n,N-n+1}}{N-n+2}\xi^n\bar{\xi}^{N-n+2}+\frac{\bar{A}_{n,N-n+1}}{N-n+2}
 \xi^{N-n+2}\bar{\xi}^n\right)\nonumber\\
&=\frac{2}{(1+\xi\bar{\xi})^2}\sum_{n=0}^{l+1}\left[A_{n,N-n+1}\left(1+{\textstyle{\frac{N-n+1}{N-n+2}}}\xi\bar{\xi}\right)\xi^n\bar{\xi}^{N-n+1}\right.\nonumber\\
&\qquad\qquad\qquad\qquad\qquad \left.+\bar{A}_{n,N-n+1}\left({\textstyle{\frac{n}{N-n+2}}}+{\textstyle{\frac{n-1}{N-n+2}}}\xi\bar{\xi}\right)\xi^{N-n+2}\bar{\xi}^{n-1}\right]\nonumber\\
&=\frac{2F}{(1+\xi\bar{\xi})^2},\nonumber
\end{align}
where the first equality comes from differentiating the explicit expression for $r$, the second equality from rearranging terms and the final equality from the definition of $F$.

The proof for $N=2l+2$ is the same with the added computation
\[
\bar{\partial}\left(\frac{2A_{l+2,l+1}\xi^{l+2}}{(1+\xi\bar{\xi})(l+2)} \bar{\xi}^{l+2}\right)
    =\frac{2{A}_{l+2,l+1}}{(1+\xi\bar{\xi})^2}(1+{\textstyle{\frac{l+1}{l+2}}}\xi\bar{\xi})\xi^{l+2}\bar{\xi}^{l+1}.
\]
For $N>0$ clearly $\xi=0$ is an umbilic point. To show that it is isolated we compute
\begin{align}
\bar{\partial}F&=(1+\xi\bar{\xi})\sum_{n=-1}^{l}(N-n)A_{n+1,N-n}\xi^{n+1}\bar{\xi}^{N-n-1}
     +{\textstyle{\frac{n(n+1)}{N-n+1}}}\bar{A}_{n+1,N-n}\xi^{N-n+1}\bar{\xi}^{n-1}\nonumber\\
&=(1+\xi\bar{\xi})\sum_{n=0}^{N}(N-n+1)A_{n,N-n+1}\xi^{n}\bar{\xi}^{N-n}\nonumber\\
&=(1+R^2)G_N(\theta)e^{-iN\theta}R^N.\nonumber
\end{align}
By assumption, the polynomial $P_N$ has no zeros on the unit circle, and so $G_N(\theta)\neq0$. Thus $R=0$ is an isolated umbilic point and
is non-degenerate of order $N$.
\end{proof}

\vspace{0.2in}

\section{Examples}

We now demonstrate direct proofs of the Main Theorem for orders $N=1,2,3,4$. For order $N\geq5$ there is no direct method. 

\vspace{0.1in}

\noindent{\bf N=1,2}:

In these cases the statement is easily seen to be trivially true. 

\vspace{0.1in}

\noindent{\bf N=3}:

Let $\zeta_1,\zeta_2$ and $\zeta_3$ be the roots of the cubic polynomial satisfying the single Lagrangian condition
\[
|\Delta_1^3|=1 .
\]
Thus
\[
{\textstyle{\frac{1}{3}}}|\zeta_1+\zeta_2+\zeta_3|=1
\]
and we conclude that $|\zeta_i|\geq 1$ for some $i$. Thus the number of roots inside the unit circle must be less than or equal to 2, as claimed.

\vspace{0.1in}

\noindent{\bf N=4}:

Let $\zeta_1,\zeta_2$,$\zeta_3$ and $\zeta_4$ be the roots of the quartic polynomial satisfying the Lagrangian conditions
\[
|\Delta_2^4|=1 \qquad\qquad \Delta_1^4=\Delta_2^4\bar{\Delta}_1^4.
\]
The first of these conditions implies that
\[
{\textstyle{\frac{1}{6}}}|\zeta_1\zeta_2+\zeta_1\zeta_3+\zeta_1\zeta_4+\zeta_2\zeta_3+\zeta_2\zeta_4+\zeta_3\zeta_4|=1,
\]
and we conclude that $|\zeta_i\zeta_j|\geq 1$ for some $i$ and $j$. This means that $|\zeta_i|\geq 1$ for some $i$, and hence the number of roots inside the unit circle must be less than or equal to 3, as claimed.

\vspace{0.1in}

\noindent{\bf N=5}:

In the case of the quintic we are considering (after substitution of the Lagrangian conditions) the family of polynomials:
\[
P_5=\zeta^5+5\alpha e^{iA}\zeta^4+10\bar{\alpha}\zeta^3+10e^{iA}\zeta^2+\beta\zeta+\gamma,
\]
for $\alpha,\beta,\gamma\in{\mathbb C}$ and $A\in{\mathbb R}$.

The Main Theorem states that such a polynomial has at most 3 roots inside the unit circle. By elementary arguments, as in the previous cases, it is not hard to show that there are at most 4 roots inside the unit circle, but reducing this to 3 roots is extremely difficult by elementary methods.

\vspace{0.2in}

\noindent{\bf Acknowledgements}:
The authors would like to thank Gerhard Schmeisser for helpful discussions during the original evolution of this work and to the anonymous Referee for pointing out the relationship with self-inversive polynomials. 

\vspace{0.2in}

\noindent{\bf Statements and Declarations}:

This research was supported by the  Research in Pairs program of the Mathematisches Forschungsinstitut Oberwolfach. The authors have no further relevant financial or non-financial interests to disclose. No data was collected in the course of this research.

\end{document}